\documentclass{amsart}

\usepackage{graphicx}
\usepackage{amssymb,latexsym}
\usepackage{amsmath,amsfonts,amstext,bbm}

\newtheorem{theorem}{Theorem}
\newtheorem{lemma}{Lemma}
\newtheorem{proposition}{Proposition}
\newtheorem{definition}{Definition}

\DeclareMathOperator{\var}{var}

\def\Xint#1{\mathchoice
{\XXint\displaystyle\textstyle{#1}} 
{\XXint\textstyle\scriptstyle{#1}} 
{\XXint\scriptstyle\scriptscriptstyle{#1}} 
{\XXint\scriptscriptstyle\scriptscriptstyle{#1}} 
\!\int}
\def\XXint#1#2#3{{\setbox0=\hbox{$#1{#2#3}{\int}$ }
\vcenter{\hbox{$#2#3$ }}\kern-.6\wd0}}

\def\dashint{\Xint-}

\begin{document}
\title[ON REGULARITY OF THE DISCRETE HARDY-LITTLEWOOD MAXIMAL FUNCTION]
      {On regularity of the discrete Hardy-Littlewood maximal function}
      
\author{Faruk Temur}
\address{Department of Mathematics\\
       Izmir Institute of Technology, Urla, Izmir, 35430, Turkey }
\email{faruktemur@iyte.edu.tr}
\date{January 8, 2016}    

\begin{abstract}
 We show that the variation of the  discrete Hardy-Littlewood maximal function is bounded. We explicitly compute the boundedness constant.  We adapt the methods of Kurka to this case, and observe that these methods are substantially shorter and simpler in this discrete case.  
\end{abstract}

\maketitle

\section{Introduction}\label{intro}

In \cite{k} J. Kinnunen proved the boundedness of the Hardy-Littlewood maximal operator given by 
\[Mf(x):=\sup_{r>0} \dashint_{B(x,r)}|f(y)|dy\]
 on the Sobolev space $W^{1,p}(\mathbb{R}^n)$ for $1<p \leq \infty.$ Since  $Mf$ is never integrable for non-trivial functions this cannot be extended to $p=1.$ However  one can ask  whether the operator $f \mapsto \nabla Mf$ is bounded from $W^{1,p}(\mathbb{R}^n)$  to $L^1(\mathbb{R}^n)$. This question, asked by Hajlasz and Onninen in  \cite{ho}, was answered  positively for $n=1$ in  the easier case of non-centered maximal function by Tanaka, and for the centered case recently by Kurka; see \cite{t,ku}. Indeed  the result of Tanaka was strengthened by J.M. Aldaz and J. P\'erez L\'azaro in \cite{apl2} to show
\begin{equation}\label{e1.1}
 \var{\widetilde{M}f\leq \var f}
 \end{equation}
where  $\widetilde{M}f$ is the non-centered maximal function, whereas Kurka derived his answer to the question from  the analogous result for the centered one:
 \begin{equation}\label{e1.2}
 \var{{M}f\leq C \var f}.
 \end{equation}

Consider the discrete Hardy-Littlewood maximal function 
\[Mf(n)=\sup_{r\in \mathbb{Z}^+}\frac{1}{2r+1}\sum_{k=-r}^r |f(n+k)|\] 
where  $f:\mathbb{Z}\mapsto \mathbb{R}$  and  $\mathbb{Z}^+$  denotes non-negative integers.  One can similarly define the non-centered version:
\[\widetilde{M}f(n)=\sup_{r,s\in \mathbb{Z}^+}\frac{1}{r+s+1}\sum_{k=-r}^s |f(n+k)|.\] 
 Although the result of Kinnunen   does not meaningfully extend to this setting, the  analogue of \eqref{e1.1}     was showed by Bober, Carneiro, Hughes and Pierce in \cite{bchp}.  In this paper we will extend   \eqref{e1.2} to discrete setting. More precisely  define
\[\var f:= \sum_{k\in \mathbb{Z}}|f(k+1)-f(k)|\]
for  a function $f:\mathbb{Z}\mapsto \mathbb{R}$. Then  we prove the following. 
\begin{theorem}
Let $f:\mathbb{Z}\mapsto \mathbb{R}$ be a function of bounded variation. Then
\[\var{{M}f\leq C \var f}.\]

\end{theorem}  
It is conjectured in \cite{bchp} that $C=1$, but as in Kurka's work we are not  able to obtain this constant. 

We will adapt ideas developed by Kurka in \cite{ku} to discrete setting to obtain our result. 
The rest of the paper is organized as follows. In the next section we will give definitions necessary for classifying local extrema, and state  a lemma that handles  the variation arising from one class  of local extrema at a time. Using these we will explain  main ideas underlying the proof and then  prove this lemma. In last three sections  the issue of putting all classes together will be dealt with.

\section{Preliminaries}

Before going to our definitions we first  note that it suffices to prove our theorem for non-negative functions. So let $f\geq 0$ be a function defined on integers with bounded variation.

\begin{definition} 
{\rm\textbf{I.}} A peak is a system of three integers $p<r<q$  satisfying $Mf(p)<Mf(r)$ and $Mf(q)<Mf(r)$.\\
{\rm\textbf{II.}} We define the variation of a peak $\mathbbm{p}=\{p<r<q\}$ by 
\[\var \mathbbm{p}=2Mf(r)-Mf(p)-Mf(q). \]
{\rm\textbf{III.}} We define the variation of a system $\mathbbm{P}$ of peaks by
\[\var \mathbbm{P}=\sum_{\mathbbm{p}\in \mathbbm{P}} \var \mathbbm{p}.\]
{\rm\textbf{IV.}} We call a peak $\mathbbm{p}$ essential if  
\[\max_{p<k<q} f(k)\leq  Mf(r)-\frac{1}{4}\var \mathbbm{p}.\]
{\rm\textbf{V.}} We define averaging operators of radius k for a non-negative integer k by 
\[A_kf(n)=\frac{1}{2k+1}\sum_{j=-k}^kf(n+j)\]
{\rm\textbf{VI.}} We define the radius of  an essential peak as 
\[\omega(r):=\max\{w>0: A_{\omega}f(r)=Mf(r)\}\]
\end{definition} 
Clearly  the last part of the definition needs further elaboration. We need to know that the set under consideration is not empty and that it contains finitely many elements.  These as well as a further property of $\omega(r)$ shall be dealt with below, but first we introduce some further notation. For $x,y \in \mathbb{Z}$ the notation $[x,y]$ will stand for integers $n$ satisfying $x \leq n \leq y$, and  we will call $[x,y]$ an interval. By length of an interval $[x,y]$ we will mean the quantity $y-x.$ We define  average of a function $f$ on an interval $[x,y]$ by
\[A_{x,y}f=\frac{1}{y-x+1}\sum_{k=x}^yf(k).\]
Now we state and prove the lemma clarifying the last part of the Definition 1.

\begin{lemma}
Let $\mathbbm{p}=\{p<r<q\}$ be and essential peak. Then $w(r)$ is well defined and satisfies 
\[r-\omega(r)<p<q<r+\omega(r).\]
\end{lemma}
\begin{proof}
First let's see that our set  is nonempty.  Since $\mathbbm{p}$ is an essential peak we have 
$f(r)<Mf(r)$.   Thus  for our set to be empty  we  must have for every $\omega\geq0$
$A_{\omega}f(r)<Mf(r).$
But then  also by definition of the maximal function we must have a strictly increasing sequence $\{\omega_j\}_{j\in \mathbb{N}}$ of  natural numbers such that
\[\lim_{j\rightarrow \infty}A_{\omega_j}f(r)=Mf(r).\]
But note that we also have
\[Mf(p)\geq \lim_{j\rightarrow \infty}A_{\omega_j}f(p)=Mf(r)>Mf(p).\]
Note that this same argument also gives that our set cannot contain infinitely many elements, hence $\omega(r)$ is well defined. 

Now note that $f(p)\leq Mf(p) <Mf(r)$ and $f(q) \leq Mf(q)<Mf(r)$, thus $p \leq r-\omega(r)<r+\omega(r)\leq q$ would imply $A_{\omega(r)}f(r)<Mf(r)$, hence at least one of  $r-\omega(r)<p$,   $q<r+\omega(r)$ is true. We assume the first one is true, and the second is wrong: the converse can be dealt with similarly.   We have $A_{p+\omega(r)-r}f(p) \leq Mf(p)<Mf(r)$, which means 
$ A_{2p+\omega(r)-r+1,r+\omega(r)}f \geq Mf(r).$ 
But  $p<2p+\omega(r)-r+1<r+\omega(r)\leq q$ means
$ A_{2p+\omega(r)-r+1,r+\omega(r)}f< Mf(r).$ 
 Hence we are done.

\end{proof}

The following is the lemma that handles variation arising from a  specific class of local extrema. As shall be explained, it plays a fundamental role in our proof.

\begin{lemma}
Let $[x,y]$ be an interval of length $L$ with $L$ an even integer. Let $\mathbbm{p_i}=\{p_i<r_i<q_i\}$ be a system of essential peaks satisfying 
\[x \leq r_1<q_1 \leq p_2<r_2 <q_2 \leq \ldots \leq p_{m-1} < r_{m-1} \leq p_m< r_m \leq y\]
and $32L<w(r_i)\leq 64L$ for $1\leq i \leq m.$ Then  there exists $s<u<v<t$ such that,
\[x-64L \leq s,\ \ t \leq y+64L, \ \ u-s \geq 4L, \ \ v-u=L, \ \ t-v \geq 4L\]
\[\min\{f(s),f(t)\}- A_{u,v}f \geq \frac{1}{12} \sum_{i=1}^m \var \mathbbm{p}_i\]

\end{lemma}
This lemma says that if in a system of essential peaks, all peaks lie in an interval of length comparable to all of their  radii, the variation of this system can be bounded by using values of the function at nearby points. So this immediately implies that we can put together such systems located at sufficiently distant intervals easily. Hence  even if we do not require the peaks to lie in an interval of certain length, the system can be broken into subsystems using a finite covering of the real line by equally spaced intervals and then easily dominated by the variation of the function. As we will see in the section 3, it is very easy to estimate the variation of non-essential peaks, so  proving this lemma reduces the problem to taking care of systems  with essential peaks of different length scales.

\begin{proof}

We shall decompose the system into three parts. If we take the first and the last peaks out, there remains a system which entirely lies in the interval. Thus proving the lemma for single peaks and systems lying in the interval, with  constant on the right hand side 1/4 instead of 1/12 suffices. We will first prove the lemma for a single peak so let $\mathbbm{p}=\{p<r<q\}$ denote our system. Our first step is to find $s,t$ satisfying
\[f(s)\geq Mf(r), \ \ \ \  x-64L \leq s \leq 2q-(r+\omega(r))\] 
\[f(t)\geq Mf(r), \ \ \ \  2p-(r-\omega(r)) \leq t \leq y+64L\] 
We shall utilize the same ideas as used in Lemma 1 and since the same procedure dealswith both, we shall find an $s$ only. 
\[A_{\omega(r)}f(r)=Mf(r), \ \ \ A_{r+\omega(r)-q}f(q)\leq Mf(q) < Mf(r)\]
thus 
\[A_{r-\omega(r),2q-(r+\omega(r))-1}f\geq Mf(r)\] 
Since $x-64L \leq r-\omega(r)$, there must be an $s$ with desired properties.

To locate  suitable $u,v$ we shall  consider two subcases:\\
\textbf{I.}If $q-p < 12L$ then  
\[s <  2q-(r+ \omega(r))< 2p +24L-r-32L < p-8L.\]
Similarly 
\[t > 2p-(r-\omega(r)) >2q-24L-r+\omega(r) > q+8L. \]
So we set $u=p-L/2, \  v=p+L/2$ if $Mf(p) \leq Mf(q)$, and $u=q-L/2,\ v=q+L/2$ otherwise. This choice clearly satisfies distance requirements, and
\[\min\{f(s), f(t)\}-A_{u,v}f \geq Mf(r)-\min\{Mf(p), Mf(q)\}\geq \frac{1}{2}\var \mathbbm{p} \] 
\textbf{II.} Let $q-p \geq 12L$. Since $\mathbbm{p}$ is an essential peak
\[f(s), f(t) \geq Mf(r)> \max_{p<k<q}f(k).\]
Thus $s\leq p<q \leq t.$ Choosing 
\[u=\left\lfloor \frac{p+q}{2} \right\rfloor-L/2, \ \ v=\left\lfloor \frac{p+q}{2} \right\rfloor+L/2\]
Then 
\[u-s\geq \frac{p+q}{2}-1-L/2-p \geq \frac{q-p}{2}-L/2-1 \geq 5L-1 \geq 4L\]
\[t-v \geq q-\frac{p+q}{2}-L/2 \geq 5L\]
and
\[\min\{f(s), f(t)\}-A_{u,v}f \geq Mf(r)- \max_{p<k<q}{f(k)} \geq \frac{1}{4}\var \mathbbm{p}.\]

Now assume that our peaks are entirely contained in $[x,y]$, so $x\leq p_1, q_{m} \leq y.$ We will work with a modification of our system: set
\begin{equation*}
\begin{aligned}
e_i&=p_i, \ \ \ \ \  \  i=1 \ \text{or} \ Mf(p_i)\leq Mf(q_{i-1})\\
e_i&=q_{i-1}, \ \ \    i=m+1 \  \text{or}  \ Mf(p_i)> Mf(q_{i-1})
\end{aligned}
\end{equation*}
and 
\[\widetilde{\mathbbm{p}_i}=\{e_i<r_i<e_{i+1}\}, \ \ 1 \leq i \leq m. \]
 We will  show the existence of $s_i, t_i$ for  $1\leq i \leq m$ that satisfy
 \[f(s_i) \geq Mf(e_{i+1})+ \frac{Mf(r_i)-Mf(e_{i+1})}{e_{i+1}-r_i}\cdot \omega(r_i), \ \ x-64L \leq s_i \leq x-30L,\]
\[f(t_i) \geq Mf(e_{i})+ \frac{Mf(r_i)-Mf(e_{i})}{r_i-e_i}\cdot \omega(r_i), \ \ y+30L \leq t_i \leq y+64L.\]
We will find  $s_i$, and $t_i$ are found similarly. We have
\[A_{\omega(r_i)}f(r_i)=(2\omega(r_i)+1)Mf(r_i),\]
\[ A_{r+\omega(r_i)-e_i}f(e_i)\leq (2(r_i+\omega(r_i)-e_i)+1)Mf(e_{i+1}).\]
So 
\[ A_{r_i-\omega(r_i),2e_i-r_i-\omega(r_i)-1}f\geq (2\omega(r_i)+1)(Mf(r_i)-Mf(e_{i+1}))+2(e_i-r_i)Mf(e_{i+1})\]
Since  $x-64L \leq r_i-\omega(r_i)$ and $   2e_i-r_i-\omega(r_i)-1 \leq 2y-x-32L =x-30L$
there exists an $s_i$ with asserted properties. 

To locate $u, v$ we consider two cases.\\
\textbf{I.} Let 
\[|Mf(e_{m+1})-Mf(e_1)|>\frac{1}{2}\sum_{i=1}^{m}\var\widetilde{\mathbbm{p_i}}.\]
Let's also assume that $Mf(e_{m+1})>Mf(e_1)$, the other case is similar. Since $Mf(e_1)\geq f(e_1)$, if we can show that $f(s_i),f(t_j) \geq Mf(e_{m+1})$ for some $i,j$ choosing $s=s_i,t=t_j,u=e_1-L/2,v=e_1+L/2$ will do. From our choice of $s_i$ we have $Mf(s_m)\geq Mf(e_{m+1})$ and 
\begin{equation*}\begin{aligned}
f(t_m)\geq Mf(e_m)+\frac{Mf(r_m)-Mf(e_{m})}{r_m-e_m}\cdot (r_m-e_m)&=Mf(r_m)\\
& \geq Mf(e_{m+1})
\end{aligned}
\end{equation*}
\textbf{II.} Let
\[|Mf(e_{m+1})-Mf(e_1)| \leq \frac{1}{2}\sum_{i=1}^{m}\var\widetilde{\mathbbm{p_i}}.\]
We know that  
\[Mf(e_{m+1})-Mf(e_{1}) = \sum _{i=1}^{m}  \big( Mf(r_{i})-Mf(e_{i}) \big) - \big( Mf(r_{i})-Mf(e_{i+1}) \big)  \]
and that
\[ \sum _{i=1}^{m} \var \widetilde{\mathbbm{p}}_{i} = \sum _{i=1}^{m} \big( Mf(r_{i})-Mf(e_{i}) \big) + \big( Mf(r_{i})-Mf(e_{i+1}) \big). \]
Thus we have
\[\sum _{i=1}^{m}  Mf(r_{i})-Mf(e_{i}), \    \sum _{i=1}^{m}  Mf(r_{i})-Mf(e_{i+1})  \geq \frac{1}{4}\sum _{i=1}^{m} \var \widetilde{\mathbbm{p}}_{i},\]
We choose $i_0,j_0$ to be the indices that maximize the expressions
\[\frac{Mf(r_{i})-Mf(e_{i+1})}{e_{i+1}-r_{i}}, \ \ \frac{Mf(r_{j})-Mf(e_{j})}{r_{j}-e_{j}}. \]
Then we have
\begin{equation*}
\begin{aligned}
f(s_{i_0})-Mf(e_{i_0+1})&\geq  \frac{Mf(r_{i_0})-Mf(e_{i_0+1})}{e_{i_0+1}-r_{i_0}}\cdot \omega(r_{i_0}) \\  &\geq \frac{Mf(r_{i_0})-Mf(e_{i_0+1})}{e_{i_0+1}-r_{i_0}}\cdot 32L \\ &\geq \frac{Mf(r_{i_0})-Mf(e_{i_0+1})}{e_{i_0+1}-r_{i_0}}\cdot32 \cdot \sum _{i=1}^{m} e_{i+1}-r_{i} \\
 &=  32\sum _{i=1}^{m} \frac{Mf(r_{i})-Mf(e_{i+1})}{e_{i+1}-r_{i}} \cdot (e_{i+1}-r_{i}) \\
 & \geq  32\sum _{i=1}^{m}  Mf(r_{i})-Mf(e_{i+1}) \\
 & \geq  8 \sum _{i=1}^{m} \var \widetilde{\mathbbm{p}}_{i}.
\end{aligned}
\end{equation*}
The same process applies to $f(t_{j_0})-Mf(e_{j_0})$. So set $s=s_{i_0},t=t_{j_0},u=e_{j_0}-L/2,v=e_{j_0}+L/2$ and note that 
\[|Mf(e_{j_0})-Mf(e_{i_0+1})| \leq \sum _{i=1}^{m} \var \widetilde{\mathbbm{p}}_{i}.\] 
Hence this choice satisfies desired properties.
\end{proof}

\section{bounding systems containing  different scales}
We first fix a system 
\[ a_1<b_1<a_2<b_2<\ldots<a_{\sigma}<b_{\sigma}<a_{\sigma+1}\]
satisfying $Mf(a_i)<Mf(b_i)$ and $Mf(a_{i+1})<Mf(b_i)$ for $1 \leq i\leq \sigma.$ We will use  $\mathbbm{P}$ to denote collection of all peaks $\mathbbm{p}_i=\{a_i<b_i<{a_{i+1}}\}$ 
arising from this system. The letter $\mathbbm{E}$ will stand for those $\mathbbm{p}_i$ that are essential. We further partition the essential peaks as follows: for $n >5, \  k \in \mathbbm{Z}$ we define
\[\mathbbm{E}^n_k=\{\mathbbm{p}_i \in \mathbbm{E}:2^{n-1}<\omega(b_i)\leq2^n, \ k2^{n-5}<b_i \leq (k+1)2^{n-5}\},\]
and we let $\mathbbm{E}'$ denote all essential peaks not belonging to one of the above collections.  

We first will  bound the variation of non-essential peaks, and then describe how to handle  $\mathbbm{E}'$. After these two relatively easy tasks we will set ourselves to bounding the variation of remaining peaks.  

\begin{lemma}
We have the inequality
\[\var ({\mathbbm{P}\setminus \mathbbm{E}}) \leq 2 \var f.\]
\end{lemma}
\begin{proof}Since $\mathbbm{p} \in \mathbbm{P}\setminus \mathbbm{E} $ is  a non-essential peak we have a point $x_i \in [a_i+1, a_{i+1}-1]$ satisfying 
\[f(x_i) \geq Mf(b_i) -\frac{1}{4} \var \mathbbm{p}_i. \]
Then we have 
\begin{equation*}
\begin{aligned}
|f(x_i)-f(a_i)|+ |f(x_{i})-f(a_{i+1})| & \geq 2f(x_i)-f(a_i)-f(a_{i+1}) \\
& \geq 2(Mf(b_i) -\frac{1}{4} \var \mathbbm{p}_i)- Mf(a_i)-Mf(a_{i+1})\\
&= \frac{1}{2}\var \mathbbm{p}_i.
\end{aligned}
\end{equation*}
From this our assertion is clear.
\end{proof}
\begin{lemma}
We have
\[\var \mathbbm{E}' \leq 1200 \cdot \var f.\]
\end{lemma}
\begin{proof}
We partition the integers into subsets $\mathbbm{Z}_l=300\mathbbm{Z}+l$ for $  0\leq l <300.$  
Similarly partition $\mathbbm{E}'$ into
\[\mathbbm{E}'_l=\{ \mathbbm{p}_i=\{a_i<b_i<a_{i+1}\}:\mathbbm{p}_i \in \mathbbm{E}', \  b_i \in  \mathbbm{Z}_l\}.\]
We apply  to $\mathbbm{p}_i$ the same procedure as in the proof of Lemma 2. for a single peak to find $s_i<u_i<t_i$ satisfying $b_i-32 \leq s_i, t_i \leq b_i+32$ and 
 \[\min\{f(s_i),f(t_i)\}-f(u_i)\geq \frac{1}{4}\var\mathbbm{p}_i.\]
Using these points we have 
\begin{equation*}
\begin{aligned}
  \var \mathbbm{E}'_l = \sum_{\mathbbm{p}_i\in\mathbbm{E}'_l} \var \mathbbm{p}_i  \leq 4\cdot \sum_{\{i:\mathbbm{p}_i\in\mathbbm{E}'_l\}} |f(s_i)-f(u_i)|+|f(t_i)-f(u_i)| \leq 4\cdot \var f  
\end{aligned}
\end{equation*} 
since the peaks in $\mathbbm{E}'_l$ are sufficiently distant.
Thus
\[\var \mathbbm{E}' = \sum_l  \var \mathbbm{E}'_l \leq 300\cdot4\cdot\var f=1200 \cdot \var f.\]
\end{proof}
To handle the remaining  peaks we need to classify further. The next lemma will serve to this purpose.

\begin{lemma}
Let $\mathbbm{E}^n_k$ be non-empty for some $n\geq 6, k\in \mathbbm{Z}$. Then one of the following is true:\\
{\rm\textbf{A.}} There exists 
$s<\alpha<\beta<\gamma< \delta<t$ satisfying 
\[(k-64)2^{n-5} \leq s, \ \  \ t \leq(k+65)2^{n-5},\]
 \[  \alpha-s \geq 2^{n-5},\ \ \ \beta-\alpha \geq 2^{n-5} \ \ \ \gamma-\beta \geq2^{n-4}, \ \ \ \delta-\gamma\geq 2^{n-5}, \ \ \ t-\delta \geq 2^{n-5} \]   
\[ \min\{f(s),f(t)\}-\max \{A_{\alpha,\beta}f, A_{\gamma,\delta}f\} \geq \frac{1}{24}\var \mathbbm{E}^n_k \]
{\rm \textbf{B.}} There exists $\alpha<\beta<u<v<\gamma<\delta$ satisfying
\[(k-64)2^n \leq \alpha, \ \  \ \delta \leq(k+65)2^n,\]
\[\beta-\alpha \geq 2^{n-5}, \ \ \ u-\beta\geq 2^{n-5}, \ \ \  v-u\geq2^{n-5}, \ \ \ \gamma-v\geq 2^{n-5}, \ \ \ \delta-\gamma\geq 2^{n-5} \]

\[\min\{A_{\alpha,\beta}f, A_{\gamma,\delta}f \}-A_{u,v}f \geq \frac{1}{24}\var\mathbbm{E}^n_k.\]
\end{lemma}
\begin{proof}
We have by  Lemma 2 points $s<u<v<t$ for peaks of $\mathbbm{E}^n_k$ and interval $[k2^{n-5},(k+1)2^{n-5}].$ We then define
\[\alpha=u-3\cdot 2^{n-5}, \ \ \ \beta=u-2\cdot 2^{n-5}, \ \ \ \gamma=v+2\cdot 2^{n-5}, \ \ \ \delta= 3\cdot 2^{n-5}.\]
If the  inequality 
\[\min\{A_{\alpha,\beta}f, A_{\gamma,\delta}f \} \geq \frac{1}{2}\min\{f(s),f(t)\}+\frac{1}{2}A_{u,v}f\]
is satisfied then we just need to subtract $A_{u,v}f$ from both sides and use the Lemma 2
 to see \textbf{B} satisfied. Assume it does not hold. We first assume $A_{\alpha,\beta}=\min\{A_{\alpha,\beta}f, A_{\gamma,\delta}f \}$. In this case 
 \[\max\{A_{\alpha,\beta},A_{u,v}\}+\frac{1}{2}\min\{f(s),f(t)\}\leq \min\{f(s),f(t)\}+\frac{1}{2}A_{u,v}f \]                                                                                                                                       Applying Lemma 2 from here yields the desired inequality if we keep $\alpha, \beta$ the same, and set $\gamma=u,\delta=v$. For the case  $A_{\gamma,\delta}=\min\{A_{\alpha,\beta}f, A_{\gamma,\delta}f \}$  all  we need is to keep $\gamma,\delta$ the same and set $\alpha=u,\beta=v.$                                 
\end{proof}

Thus we define $\mathcal{A}$ to be the union of $\mathbbm{E}_k^n$ satisfying \textbf{A}, and $\mathcal{B}$ as the union those satisfying \textbf{B}. We further define $\mathcal{A}_K^n$ to be the union of $\mathbbm{E}_k^n$ in $\mathcal{A}$ for which $k=\mod 300$, and   $\mathcal{B}_K^n$ is defined analogously. Notice that since we took a finite number of peaks in  , there exists  $n_A$ representing the largest $n$ for which $\mathcal{A}_K^n$ is non-empty for at least one $K$. Similarly we have an $n_B.$ 
 In the next two sections we shall  deal respectively with variation arising from peaks of $\mathcal{A}$ and $\mathcal{B}$.

 \section{The Variation of peaks of $\mathcal{A}$ }
 
 The following is the main proposition we want to prove in this section.
 \begin{proposition}
Let $0\leq N \leq 11$,   $0\leq K\leq 299$ and let $L_N$ denote $2^{N-6}$ if $N\geq 6$ and $2^{N+6}$ if $n\leq5.$  There exists a system
\[x_1<u_1<v_1<x_2<u_2<v_2<x_3<\ldots<x_m<u_m<v_m<x_{m+1}\]   
 with properties
 \[u_i-x_i\geq L_{N}, \ \ v_i-u_i\geq L_N, \ \ x_{i+1}-v_{i}\geq L_N, \ \ 1 \leq i\leq m,   \]
\[\sum_{i=1}^mf(x_i)+f(x_{i+1})-2A_{u_i,v_i}f\geq \frac{1}{60}\sum_{n=N\mod 12}\var{\mathcal{A}_{K}^n}.\]  
 \end{proposition} 
We shall prove this inductively. Let $n_N=N \mod 12$ denote the maximum integer $n$ for which $\mathcal{A}_{K}^n$ is non-empty.  We clearly have a system as described above that bounds the variation of $\mathcal{A}_{K}^n$ which have $2^{n_N-5}$  instead of $L_N$ -this   is true only if $n_N>N$ of course, but if $n_N=N$ we directly obtain the desired system using Lemma 2-. Now assume we have a system that bounds sum of variations coming from all classes $\mathcal{A}_K^n$ for $n>n_0$  where $\mathcal{A}_K^{n_0}$ is non-empty, and that has $2^{n_0+12}$ instead of $L_N$. If we can modify this system so that it bouns all classes for $n\geq n_0$ with $2^{n_0}$  replacing $L_N$, a finite iteration would give our proposition.

Thus we assume there exists a system 
\begin{equation}\label{eqs}
x_1<u_1<v_1<x_2<u_2<v_2<x_3<\ldots<x_{m_0}<u_{m_0}<v_{m_0}<x_{m_0+1}\end{equation} 
 that satisfies conditions given by the inductive hypothesis above.
The class $\mathcal{A}_K^{n_0}$  is a union of a finite number of systems of peaks $\mathbbm{E}_k^{n_0}$, we will describe how to incorporate these into the existing system. Pick one  such  $\mathbbm{E}_k^{n_0}$   and consider  $s<\alpha<\beta<\gamma<\delta<t$ 
coming from the alternative \textbf{A} of Lemma 5 for it. We will modify \eqref{eqs} according to  its relation with the interval $[s,t]$. 

\textbf{I.} First assume for any  $1\leq i\leq m_0$  we have $\text{dist}([u_i,v_i],[s,t])\geq2^{n-5}.$  In this case  one of the intervals
\[(-\infty,u_1],\ [v_1,u_2],\ \ldots\  [v_{m_0-1},u_{m_0}],\ [v_0,\infty)\]
must contain $[s,t]$. This interval also contain a unique $x_i, 1\leq i \leq m_0$, which must satisfy either $\text{dist}(x_i,\beta)\geq \text{dist}(x_i,\gamma)$ or $\text{dist}(x_i,\beta)< \text{dist}(x_i,\gamma)$ . If the first happens we take $s,\alpha, \beta$, otherwise we take $\gamma,\delta,t$  and add them to our system. The new system is easily seen to 
satisfy desired properties.

\textbf{II.}  There exists an $i$ with $\text{dist}([u_i,v_i],[s,t])<2^{n-5}$. We first note that this $i$ is  unique. Observe that either $(k-150)2^{n-5}\in [u_i,v_i]$ or $(k+150)2^{n-5}\in [u_i,v_i],$ we will assume the first, as the second is handled similarly. Let $g=h=k \mod 300$ be such that 
$(g-155)2^{n-5} \leq x' <(g+145)2^{n-5},(h-155)2^{n-5} \leq y' <(h+145)2^{n-5}.$ Notice that these condition determine $g,h$ uniquely. Using these we partition our interval 
\[[u_i,(g+150)2^{n-5}-1], \  [(g+150)2^{n-5},(g+450)2^{n-5}-1], \ldots  [(h-150)2^{n-5},v_i].\]
One of these subintervals contains $(k-150)2^{n-5}$  which will be denoted by $I$ and, average of $f$ on one of these subintervals is less than or equal to  average over $[u_i,v_i]$, we will call this  $[u_i',v_i'].$ If $I$ is not the same as $[u_i',v_i']$, then  this latter interval is distant enough from $[s,t]$, and replacing $[u_i,v_i]$ by $[u_i',v_i']$  and choosing appropriate ones out of $\{s,\alpha,\beta,\gamma,\delta,t\}$ 
will do. 
If they are  the same then  we have to consider two different cases. Either   
 there exists $[c,d]\subset[u_i',v_i']$ with  $d-c\geq 2^{n-5}$ such 
that

\[A_{c,d}f \leq A_{u_i,v_i}f-\frac{1}{120}\var \mathbbm{E}_k^{n_0},\]
or we have a system 
$c<d<y<c'<d'$ with $[c,d']\subset[u_i',u_i'+300\cdot2^{n-5}]$ and 
\[d-c\geq 2^{n-5}, \ \ \ y-d \geq 2^{n-5}, \ \ \  c'-y\geq 2^{n-5}, \ \ \  d'-c'\geq 2^{n-5}, \] 
 such that 
\[A_{c,d}f +
A_{c',d'}f-f(y) \leq A_{u_i,v_i}f-\frac{1}{120}\var \mathbbm{E}_k^{n_0}.\]
 In both cases  what to do is clear,  in the first case $[u_i,v_i]$ is replaced by $[c,d]$, while in the second we replace $[u_i,v_i]$ by two intervals $[c,d],[c',d']$ and the point $y$ between them. But that one of these must hold should be shown. We set
\[w_i'=u_i'+\left\lceil \frac{v_i'-u_i'}{5}\right\rceil \] 
and observe that both $[u_i',w_i'], \ [w_i'+1,v_i']$ are longer than $2^{n-5}$. We have either
\begin{equation}\label{eqd} A_{w_i'+1,v_i'}f\leq A_{u_i',v_i'}f-\frac{1}{120}\var\mathbbm{E}_k^{n_0} \ \ \text{or} \ \ A_{u_i',w_i'}f \leq A_{u_i',v_i'}+\frac{4}{120}\var\mathbbm{E}_k^{n_0},  
\end{equation}  
and if the first holds we just set $c=w_i'+1,d=v_i'$ to obtain \textbf{a} whereas   if the second holds we let $c=u_i',d=w_i',y=s,c'=\alpha,d'=\beta.$ That $y-d\geq2^{n-5}$ follows from the definitions of $u_i',w_i'.$ 

We thus incorporated the  first $\mathbbm{E}^{n_0}_k$ into the system. For the rest we apply  a similar procedure but, we also have to deal with  previously made changes, which shorten the distance between successive points from $2^{n+7}$ to $2^{n-5}$. Let us incorporate  a second system $\mathbbm{E}^{n_0}_l$. Let our modified system be
\begin{equation}\label{eqs2}
x_1<u_1<v_1<x_2<u_2<v_2<x_3<\ldots<x_{m_{0,1}}<u_{m_{0,1}}<v_{m_{0,1}}<x_{m_{0,1}+1}
\end{equation} 
and consider $s'<\alpha'<\beta'<\gamma'<\delta'<t'$  coming from  Lemma 5 for $\mathbbm{E}^{n_0}_l$. We again have the same two alternatives which this time we will call \textbf{I'},\textbf{II'}, and  if \textbf{I'} is the case, exactly same ideas suffice. If on the other hand $\text{dist}([u_i,v_i],[s',t'])\leq 2^{n-5}$ holds for some $1\leq i\leq m_{0,1}$, then some additional consideration is needed. First we need to see that this  $[u_i,v_i]$ is unique. Since $k\neq l$ we have $\text{dist}([s',t'],[(k-150)2^{n-5}(k+150)2^{n-5}])\geq 2^{n-5}$ such $[u_i,v_i]$ can be either  unmodified intervals, or  only first of three types of intervals arising from \textbf{II}. If $[u_i,v_i]$ is close to an unmodified interval, it is sufficiently distant from all other unmodified intervals and intervals arising from $\textbf{II}.$ 
Similarly being close to an interval arising from $\textbf{II}$ guarantees distance from all unmodified intervals. Thus $[u_i,v_i]$ is unique. After this methods described in \textbf{II} handles both cases.  Clearly these considerations suffice to add the remaining systems, and after a finite number of steps we will have $\mathcal{A}^{n_0}_K$ incorporated.
  
  Thus the proof of our proposition is complete. From this proposition we easily deduce that  \[ \var \mathcal{A} \leq 120\cdot2^{12}\cdot 300 \cdot \var f.\]
  
  \section{The variation of peaks of $\mathcal{B}$}
  Arguments of this section will largely be analogous to those of section 4. We state the main proposition of this section.
   \begin{proposition}
Let $0\leq N \leq 11$,   $0\leq K\leq 299$ and let $L_N$ denote $2^{N-6}$ if $N\geq 6$ and $2^{N+6}$ if $n\leq5.$  There exists a system
\[x_1<y_1<u_1<v_1<x_2<y_2<u_2<v_2<\ldots<u_m<v_m<x_{m+1}<y_{m+1}\]   
 with properties
 \[y_i-x_i\geq L_N, \ \ 1 \leq i\leq m+1, \]
 \[u_i-y_i\geq L_{N}, \ \ v_i-u_i\geq L_N, \ \ x_{i+1}-v_{i}\geq L_N  \ \ 1 \leq i\leq m,   \]
\[\sum_{i=1}^mA_{x_i,y_i}f+A_{x_{i+1},y_{i+1}}f-2A_{u_i,v_i}f\geq \frac{1}{60}\sum_{n=N\mod 12}\var{\mathcal{B}_{K}^n}.\]  
 \end{proposition} 
  We shall again utilize induction. Assume we have a system 
  \[x_1<y_1<u_1<v_1<x_2<y_2<u_2<v_2<\ldots<u_{m_0}<v_{m_0}<x_{m_0+1}<y_{m_0+1}\] 
  that bounds the variation of all classes  $\mathcal{B}_K^n$ for $n>n_0$  where $\mathcal{B}_K^{n_0}$ is non-empty, and that has $2^{n_0+12}$ instead of $L_N$. 
  Let  $\mathbbm{E}_k^{n_0}$ be one of subsystems comprising $\mathcal{B}_K^{n_0}$, and 
consider  $\alpha<\beta<u<v<\gamma<\delta$ 
coming from the alternative \textbf{B} of Lemma 5 for it. We  again will  investigate the relation of our system  with the interval $[\alpha,\delta]$, this time however, we will have three cases.

\textbf{I.} First assume for all  $1\leq i\leq m_0$  we have $\text{dist}([u_i,v_i],[\alpha,\delta])\geq2^{n-5}$, and for all $1\leq i\leq m_0+1$ we have $\text{dist}([x_i,y_i],[\alpha,\delta])\geq2^{n-5}$.  This case is easy, we just choose two appropriate ones out of three intervals $[\alpha,\beta],[u,v],[\gamma,\delta]$, and incorporate to our system.  

\textbf{II.}There exist an $1\leq i\leq m_0$ such that  $\text{dist}([u_i,v_i],[\alpha,\delta])< 2^{n-5}.$ Clearly this $i$ is unique, moreover $[\alpha.\delta]$ is distant from $[x_i,y_i]$ type intervals. This case will be dealt with in the same way as the case \textbf{II} of section 4. We divide $[u_i,v_i]$ into subintervals and pick $I$, $[u_i',v_i']$ exactly in the same way.   The case when they are not the same is easy and handled as before, whereas if they are  the same  either   
 there exists $[c,d]\subset[u_i',v_i']$ with  $d-c\geq 2^{n-5}$ such 
that

\begin{equation}\label{a2}A_{c,d}f \leq A_{u_i,v_i}f-\frac{1}{120}\var \mathbbm{E}_k^{n_0},
\end{equation}
or we have a system 
$c<d<x<y<c'<d'$ with $[c,d']\subset[u_i',u_i'+300\cdot2^{n-5}]$ and 
\[d-c\geq 2^{n-5}, \ \ \ x-d \geq 2^{n-5}, \ \ \ y-x\geq 2^{n-5},\ \ \ c'-y\geq 2^{n-5}, \ \ \  d'-c'\geq 2^{n-5}, \] 
 such that 
\begin{equation}\label{b2}A_{c,d}f +
A_{c',d'}f-A_{x,y}f \leq A_{u_i,v_i}f-\frac{1}{120}\var \mathbbm{E}_k^{n_0}.
\end{equation}
In each  what to do is clear, we will show  that one of these holds. Defining  $w'_i$ as before  we have the dichotomy given in \eqref{eqd}. If the first alternative of this dichotomy holds we set $c=w_i'+1,d=v_i'$ and get \eqref{a2}, while if the second holds we set
\begin{equation}\label{choice}c=u_i', \ \ d=w_i', \  \ x=u,\ \ y=v, \ \ c'=\gamma, \ \ d'=\delta
\end{equation}
and obtain \eqref{b2}.

\textbf{III.} There exist an $1\leq i\leq m_0+1$ such that  $\text{dist}([x_i,y_i],[\alpha,\delta])< 2^{n-5}.$  This case is similar to what we have above, only essential difference will be changes in signs of averages over intervals. As above  this $i$ is unique, further $[\alpha.\delta]$ is distant from $[u_i,v_i]$ type intervals. We subdivide $[x_i,y_i]$ the way we did $[u_i,v_i]$ above and, choose $I$. This time, however, $[x_i',y_i']$ will be the subinterval on which average is not smaller than the average over $[x_i,y_i].$ If these are not the same, replacing $[x_i,y_i]$ with $[x_i',y_i']$ will suffice. If they are the same we either have $[c,d]\subset[x_i',y_i']$ with  $d-c\geq 2^{n-5}$ such 
that
\[A_{c,d}f \geq A_{x_i,y_i}f+\frac{1}{120}\var \mathbbm{E}_k^{n_0},\]
or we have a system 
$c<d<\mu<\nu<c'<d'$ with $[c,d']\subset[x_i',x_i'+300\cdot2^{n-5}]$ and 
\[d-c\geq 2^{n-5}, \ \ \ \mu-d \geq 2^{n-5}, \ \ \ \nu-\mu \geq 2^{n-5},\ \ \ c'-\nu \geq 2^{n-5}, \ \ \  d'-c'\geq 2^{n-5}, \] 
 such that 
\[A_{c,d}f +
A_{c',d'}f-A_{\mu,\nu}f \geq A_{x_i,y_i}f+\frac{1}{120}\var \mathbbm{E}_k^{n_0},\]
\[A_{c,d}f,A_{c',d'}f\geq 
A_{\mu,\nu}f\]
This last additional property handles problems arising when  $i=1$ and  $i=m_0+1$. As before in the  first case $[c,d]$ replaces $[x_i,y_i]$, while in the second $[c,d],[\mu,\nu],[c',d']$,  does. 
Defining  
\[z_i'=x_i'+\left\lceil \frac{y_i'-x_i'}{5}\right\rceil \] 
  we have the dichotomy
\[A_{z_i'+1,y_i'}f\geq A_{x_i',y_i'}f+\frac{1}{120}\var\mathbbm{E}_k^{n_0} \ \ \text{or} \ \ A_{u_i',z_i'}f \geq A_{x_i',y_i'}f-\frac{4}{120}\var\mathbbm{E}_k^{n_0}. \]
If the first is the case  we just set  $c=z_i'+1$, $d=y_i'$  to obtain \eqref{a2}, if the second holds we set $\mu=u,\ \ \nu=v,\ \ c'=\gamma, \ \ d'=\delta,$ and
\[c=\alpha, \ \ d=\beta \ \ \text{if} \ \ A_{\alpha,\beta}f\geq A_{x'_i,y_i'}f, \]
\[c=x_i',\  d=y_i'\ \ \text{if}  \ \  A_{x'_i,y_i'}f >A_{\alpha,\beta}f.\]
Here using  the interval on which average is greater guarantees the last additional property.

We thus incorporated $\mathbbm{E}_k^{n_0}$ into our system. To incorporate the rest we have to deal with previously made changes. Let us incorporate  a second system $\mathbbm{E}^{n_0}_l$. Let our modified system be
\begin{equation}\label{eqs3}
x_1<y_1<u_1<v_1<x_2<u_2<\ldots<u_{m_{0,1}}<v_{m_{0,1}}<x_{m_{0,1}+1}<y_{m_{0,1}+1}
\end{equation} 
and $\alpha'<\beta'<u'<v'<\gamma'<\delta'$  coming from  Lemma 5 for $\mathbbm{E}^{n_0}_l$. We  have  the same three  alternatives which  we will call \textbf{I'},\textbf{II'},\textbf{III'}, and  if \textbf{I'} is the case, exactly same ideas suffice. If $II'$ is the case, that is if $[\alpha,\delta]$ is close to $[u_i,v_i]$ for some $1\leq i \leq m_0$, then this is either an unmodified interval, or emerges  as the first  of three types of intervals arising from \textbf{II}. In either case $i$ should be unique by  the same considerations as in section 4, and methods explained in \textbf{II} handles this case. If  \textbf{III'} holds then by the same arguments  $[\alpha,\delta]$ is close to $[x_i,y_i]$ for a unique $1\leq i\leq m_{0}+1$, and this $[x_i,y_i]$ is either unmodified, or a result of a modification through first of three methods described in \textbf{III}. In either case methods of \textbf{III} deals with this case. Clearly these considerations suffice to add the remaining systems, and after a finite number of steps we will have $\mathcal{B}^{n_0}_K$ added.
 
 This completes the proof of our proposition from which we easily obtain 
\[\var \mathcal{B} \leq 120\cdot2^{12}\cdot 300 \cdot \var f.\]

 \section{Proof of Theorem 1}
 
 We now use results we proved in sections 3,4,5 to prove our theorem. We have
 \[\var Mf \leq \sum_{k=-\infty}^{\infty}|f(k+1)-f(k)| \leq \sup_{m,n:\ m\leq n} \sum_{k=m}^n |f(k+1)-f(k)|. \]
 For each couple $m,n$ with $m\leq n$ dispensing with redundant elements the interval $[m,n+1]$ gives a system 
 \[ b_0\leq a_1<b_1<a_2<b_2<\ldots<a_{\sigma}<b_{\sigma}<a_{\sigma+1}\leq b_{\sigma+1}\]
 with $Mf(a_i)<Mf(b_i),  \ Mf(a_{i+1})<Mf(b_i)$ for $1\leq i \leq\sigma $, and $Mf(a_1)\leq Mf(b_0),\ \ Mf(a_{\sigma+1})\leq Mf(b_{\sigma+1})$
 \begin{align*}\sum_{k=m}^n |f(k+1)-f(k)|&= \sum_{i=1}^{\sigma}\big(2Mf(b_i)-Mf(a_{i+1})-Mf(a_i)\big)\\ &+Mf(b_0)-Mf(a_1)+Mf(a_{\sigma+1})-Mf(b_{\sigma+1})\end{align*}
 We apply Lemma 3, Proposition 1, Proposition 2 to obtain 
 \[\sum_{i=1}^{\sigma}\big(2Mf(b_i)-Mf(a_{i+1})-Mf(a_i)\big)\leq (2\cdot120\cdot2^{12}\cdot300+2)\cdot \var f.\]
 On the other hand
 \begin{align*}Mf(b_0)-Mf(a_1)+Mf(a_{\sigma+1})-Mf(b_{\sigma+1})&\leq 2\sup_{k\in\mathbbm{Z}} Mf(k)-2\inf_{k\in\mathbbm{Z}}Mf(k)\\ &\leq 2\sup_{k\in\mathbbm{Z}} f(k)-2\inf_{k\in\mathbbm{Z}}f(k)\\ &\leq 2\var f
 \end{align*}
 So finally  taking supremum on the left we have
 \[\var Mf \leq (2\cdot120\cdot2^{12}\cdot300+4)\cdot \var f.\]

\end{document}